\documentclass[12pt]{article}
\usepackage{amssymb,amsmath,amsthm,mathrsfs}
\usepackage[usenames]{color}
\usepackage{hyperref}
\usepackage[all]{xy}
\usepackage{xypic}
\usepackage{graphicx}
\usepackage{amscd}
\usepackage{xcolor}
\usepackage{enumerate}

\baselineskip 18pt \textwidth 16cm  \theoremstyle{plain}

\newtheorem{theorem}{Theorem}[subsection]
\newtheorem{proposition}[theorem]{Proposition}
\newtheorem{corollary}[theorem]{Corollary}
\newtheorem{lemma}[theorem]{Lemma}
\newtheorem{definition}[theorem]{Definition}
\newtheorem{notation}[theorem]{Notation}

\newtheorem{example}[theorem]{Example}
\newtheorem{remark}[theorem]{Remark}

\newtheorem{thmdef}[theorem]{Theorem-Definition}

\newcommand{\Z}{\mathbb Z}
\newcommand{\F}{\mathbb F}

\newcommand{\nir}[1]{{\color{red}{Nir: #1}}}

\newcommand{\RamiAlt}[1]{}
\newcommand{\RamiAbout}[1]{}

\newcommand{\lbl}[1]{\label{#1}}

\DeclareMathOperator{\GL}{GL}

\DeclareMathOperator{\val}{val}
\DeclareMathOperator{\ac}{ac}

\DeclareMathOperator{\Hom}{Hom}

\DeclareMathOperator{\ad}{ad}

\DeclareMathOperator{\Res}{Res}
\DeclareMathOperator{\Lie}{Lie}

\DeclareMathOperator{\rk}{rk}

\DeclareMathOperator{\Spec}{Spec}

\DeclareMathOperator{\Def}{Def}

\DeclareMathOperator{\Gal}{Gal}

\DeclareMathOperator{\VF}{{VF}}
\DeclareMathOperator{\VG}{{VG}}

\DeclareMathOperator{\RF}{{RF}}

\DeclareMathOperator{\Jet}{Jet}
\DeclareMathOperator{\Sch}{Sch}
\DeclareMathOperator{\Set}{Set}

\newcommand{\Q}{\mathbb Q}

\DeclareMathOperator{\Irr}{Irr}

\newcommand{\A}{\mathbb A}

\definecolor{dblue}{RGB}{6,69,173}
\definecolor{lblue}{RGB}{11,0,128}

\newcommand{\colorlinks}{true}

\newcommand{\linkcolor}{lblue}
\newcommand{\citecolor}{green}
\newcommand{\urlcolor}{dblue}

\newcommand{\linkbordercolor}{red}
\newcommand{\citebordercolor}{green}
\newcommand{\urlbordercolor}{cyan}

\usepackage{hyperref}
\hypersetup{colorlinks=\colorlinks,linkbordercolor=\linkbordercolor, linkcolor=\linkcolor, citecolor=\citecolor, urlcolor=\urlcolor,urlbordercolor=\urlbordercolor,citebordercolor=\citebordercolor}%

\newcommand{\hrefHid}[2]{
\hypersetup{urlbordercolor={1 1 1}}%
\hypersetup{urlcolor=black}%
\href{#1}{#2}%
\hypersetup{urlbordercolor=\urlbordercolor}%
\hypersetup{urlcolor=\urlcolor}%
}

\newcommand{\inhref}[2]{\hyperref[#1]{#2}}

\newcommand{\inhrefHid}[2]{%
\hypersetup{linkbordercolor={1 1 1}}%
\hypersetup{linkcolor=black}%
\inhref{#1}{#2}%
\hypersetup{linkbordercolor=\linkbordercolor}%
\hypersetup{linkcolor=\linkcolor}%
}

\newcommand{\defHref}[3]{\newcommand{#1}[1][#3]{\href{#2}{##1}}}
\newcommand{\defInhref}[3]{\newcommand{#1}[1][#3]{\inhref{#2}{##1}}}
\newcommand{\defHrefHid}[3]{\newcommand{#1}[1][#3]{\hrefHid{#2}{##1}}}
\newcommand{\defInhrefHid}[3]{\newcommand{#1}[1][#3]{\inhrefHid{#2}{##1}}}

\newcommand{\defHrefBoth}[3]{%
\expandafter\defHrefHid \csname #3Hid\endcsname {#1}{#2}%
\expandafter\defHref \csname #3Vis\endcsname {#1}{#2}%
}
\newcommand{\defInhrefBoth}[3]{%
  \expandafter\defInhrefHid \csname #3Hid\endcsname {#1}{#2}%
  \expandafter\defInhref \csname #3Vis\endcsname {#1}{#2}%
}

\newcommand{\defHrefBothVis}[3]{%
\defHrefBoth{#1}{#2}{#3}%
\expandafter\defHref \csname #3\endcsname {#1}{#2}%
}
\newcommand{\defInhrefBothVis}[3]{%
  \defInhrefBoth{#1}{#2}{#3}%
  \expandafter\defInhref \csname #3\endcsname {#1}{#2}%
}

\newcommand{\defHrefBothHid}[3]{%
\defHrefBoth{#1}{#2}{#3}%
\expandafter\defHrefHid \csname #3\endcsname {#1}{#2}%
}
\newcommand{\defInhrefBothHid}[3]{%
  \defInhrefBoth{#1}{#2}{#3}%
  \expandafter\defInhrefHid \csname #3\endcsname {#1}{#2}%
}

\baselineskip 18pt \textwidth 16cm  \theoremstyle{plain}

\newtheorem*{theorem*}{Theorem}
\newtheorem*{notation*}{Notation}
\newtheorem*{definition*}{Definition}
\newtheorem*{remark*}{Remark}
\newtheorem*{conjecture*}{Conjecture}

\newtheorem{introtheorem}{Theorem}

\newtheorem{romantheorem}{Theorem}

\newtheorem{romancorollary}[romantheorem]{Corollary}

\DeclareMathOperator{\Syl}{Syl}
\DeclareMathOperator{\CH}{CH}

\renewcommand{\dim}{{\operatorname{dim}}}

\renewcommand{\Hom}{{\operatorname{Hom}}}

\title{Counting points of schemes over finite rings and counting representations of arithmetic lattices}

\author{Avraham Aizenbud and Nir Avni}
\usepackage{varioref}

\begin{document}

\maketitle

\begin{abstract}
We relate the singularities of a scheme $X$ to the asymptotics of the number of points of $X$ over finite rings. This gives a partial answer to a question of Mustata. We use this result to count representations of arithmetic lattices. More precisely, if $\Gamma$ is an arithmetic lattice whose $\mathbb{Q}$-rank is greater than one, let $r_n(\Gamma)$ be the number of irreducible $n$-dimensional representations of $\Gamma$ up to isomorphism. We prove that there is a constant $C$ (for example, $C=746$ suffices) such that $r_n(\Gamma)=O(n^C)$ for every such $\Gamma$. This answers a question of Larsen and Lubotzky.
\end{abstract}

\setcounter{tocdepth}{2}
\tableofcontents

\section{Introduction}
\subsection{Main results}

This paper has two main results. The first is about counting points of schemes over finite rings. In the formulation of this result, we use the notions of local complete intersection and rational singularities, which we recall in \S\S\ref{subsec:sing}.

\begin{introtheorem} \label{thm:intro.count} Let $X$ be a scheme of finite type over $\mathbb{Z}$. Assume that the generic fiber $X_\mathbb{Q}:=X\times_{\Spec \mathbb{Z}}\Spec \mathbb{Q}$ of $X$ is reduced, absolutely irreducible, and a local complete intersection. The following conditions are equivalent:
\begin{enumerate}
\item For any $m$,
\[
\lim_{p \rightarrow \infty} \frac{|X(\mathbb{Z}/p^m)|}{p^{m\cdot \dim X_\mathbb{Q}}} =1.
\]
\item There is a finite set $S$ of prime numbers {and a constant C} such that
\[
\left| \frac{|X(\mathbb{Z}/p^m)|}{p^{m\cdot \dim X_\mathbb{Q}}}-1\right|<C p^{-1/2},
\]
for any prime $p\notin S$ and any $m \in \mathbb{N}$.
\item \label{item:intro.count.3} There is a finite set $S$ of prime numbers {and a constant C} such that
\[
\left|\frac{|X(\mathbb{Z}/p^m)|}{p^{m\cdot \dim X_\mathbb{Q}}}-\frac{|X(\mathbb{Z}/p)|}{p^{\dim X_\mathbb{Q}}}\right|<C p^{-1},
\]
for any prime $p\notin S$ and any $m \in \mathbb{N}$.
\item \label{item:intro.count.4} $X_{{\mathbb{Q}}}$ has rational singularities.
\end{enumerate}
\end{introtheorem}
In fact, we prove this theorem also for other rings, see Theorem \ref{thm:size.X}.

We use Theorem \ref{thm:intro.count} to prove our second main result, counting representations of arithmetic groups of high rank (for example $SL_d(\Z)$ for $d \geq 3$):

\begin{introtheorem} \label{thm:intro.bnd} (see Theorem \ref{thm:introduction.bound}) There is a constant $C$ (for example, $C=747$ suffices) such that, if $G$ is an algebraic group scheme whose generic fiber $G_\mathbb{Q}$ is simple, connected, simply connected, and of $\mathbb{Q}$-rank at least two, then
\[
\left| \left\{ \parbox{2.5in}{Irreducible representations of $G(\mathbb{Z})$ of dimensions at most $n$} \right\} \right| = O \left( n^C \right),
\]
\end{introtheorem}

Theorem \ref{thm:intro.bnd} (and its generalization Theorem \ref{thm:introduction.bound} below) give an affirmative answer to a question of Larsen and Lubotzky (see \cite[\S11]{LL}).

\subsection{Counting points}

\subsubsection{Background and results}
Let $X$ be a scheme of finite type over $\mathbb{Z}$. We will be interested in the number of $\mathbb{Z} / N$-points of $X$, as a function of $N$. If $s\in \mathbb{C}$ has sufficiently large real part, then the series
\begin{equation} \label{eq:P.defn}
\mathcal{P}_X(s)=\sum_{N=1} ^{\infty} |X(\mathbb{Z} / N)| \cdot N^{-s}
\end{equation}
(where $|X(\mathbb{Z} /1)|=1$) converges absolutely and defines a holomorphic function. The domain of absolute convergence of the series \eqref{eq:P.defn} has the form $\left\{ s \mid \Re(s)> \alpha_X \right\}$. We call $\alpha_X$ the abscissa of convergence of $\mathcal{P}_X$. We have $$
\alpha_{X}=\limsup_{n \rightarrow \infty} \frac{\log \left(|X(\mathbb{Z} / 1)|+\cdots+|X(\mathbb{Z} / n)|\right)}{\log n}.$$

The Chinese Remainder Theorem implies a decomposition $\mathcal{P}_X(s)=\prod_p \mathcal{P}_{X,p}(s)$, where the product is over the set of prime numbers, and each local component is defined as
\[
\mathcal{P}_{X,p}(s)=\sum_{n=0}^{\infty} |X(\mathbb{Z} / p^n)| \cdot p^{-ns}.
\]

The functions $\mathcal{P}_{X,p}(s)$ were introduced by Borevich and Shafarevich, who conjectured that they are rational functions of $p^{-s}$. This conjecture was proved by Igusa (for hypersurfaces) and Denef (in general). For hypersurfaces, the function $\mathcal{P}_{X,p}$ is studied via the following variant of it, called the Igusa zeta function:
\[
\mathcal{Z}_{X,p}(s):=(1-p^s)\mathcal{P}_{X,p}(s+\dim X_\mathbb{Q} +1)+p^s.
\]
See \cite[\S1.3]{Mus_zeta} for a definition of $\mathcal{Z}_{X,p}(s)$ as a $p$-adic integral.

du Sautoy and Grunewald studied more general $p$-adic integrals and defined the following Euler product:
\[
\mathcal{Z}_X(s)=\prod_p \frac{\mathcal{Z}_{X,p}(s)}{\mathcal{Z}_{X,p}(\infty)}.
\]
In \cite{Mus_zeta}, Mustata asks for relations between the analytic properties of $\mathcal Z_X(s)$ and algebro-geometric properties of $X$. A partial answer is given by the following corollary of Theorem \ref{thm:intro.count}:

\begin{romantheorem}\label{thm:igu} (see \S\S\ref{ssec:pf.thm.igu}) Let $X$ be a scheme of finite type over $\mathbb{Z}$, whose generic fiber $X_\mathbb{Q}$ is reduced, absolutely irreducible, a local complete intersection, and has rational singularities. Then there is a finite set $S$ of primes such that \begin{enumerate}
\item \label{item:P.p.pole} For every $p\notin S$, either $\mathcal{P}_{X,p}(s)=1$, or the abscissa of convergence of $\mathcal{P}_{X,p}(s)$ is $\dim X_\mathbb{Q}$ and $\mathcal{P}_{X,p}(s)$ has a simple pole at $\dim X_\mathbb{Q}$.
\item \label{item:P.zeta} Let $Y=X \times_{\Spec \mathbb{Z}} \Spec S ^{-1} \mathbb{Z}$. The abscissa of convergence of $\mathcal{P}_{Y}(s)$ is $\dim X_\mathbb{Q}+1$. Moreover, if $\zeta(s)$ denotes the Riemann zeta function, then the function $\mathcal{P}_{Y}(s)/ \zeta(s-\dim X_\mathbb{Q})$ can be analytically continued to $\left\{ s \mid \Re(s)>\dim X_\mathbb{Q} +1/2 \right\}$. In particular, $\mathcal{P}_Y(s)$ has meromorphic continuation to $\Re(s)>\dim X_\mathbb{Q} +1/2$, and the only pole of the continued function on the line $\Re(s)=\dim X_\mathbb{Q} +1$ is a simple pole at $\dim X_\mathbb{Q} +1$.
\item\label{it:3} Denote $H_Y(n):=n^{-\dim Y_Q}|Y(\mathbb{Z} / n)|$. The Cezaro mean
\[
\lim_{N \rightarrow \infty} \frac{H_{Y}(1)+\cdots +H_{Y}(N)}{N}
\]
exists and is a positive real number.
\item \label{item:Z.zeta} If $X$ is a hypersurface, then the abscissa of convergence of $\mathcal{Z}_{Y}(s)$ is $0$ and the function $\mathcal{Z}_Y(s)/ \zeta(s+1)$ can be analytically continued to $\left\{ s \mid \Re(s)>-1/2 \right\}$. In particular, $\mathcal{Z}_Y(s)$ has meromorphic continuation to $\Re(s)>-1/2$ with a simple pole at $0$.
\end{enumerate}
\end{romantheorem}

\begin{remark} \label{rem:lci.implies.S.empty}
$ $
 \begin{itemize}
\item If $X$ itself is a local complete intersection, then we can take $S$ to be empty (so $Y=X$); see \S\S\ref{ssec:pf.thm.igu}.
\item Contrary to \eqref{it:3}, the sequence $H_{Y}(N)$ might not converge. This happens even for elliptic curves.
\end{itemize}
\end{remark}

\subsubsection{Sketch of the proof of Theorem \ref{thm:intro.count}}


We sketch the proof of the implication $\eqref{item:intro.count.4} \implies \eqref{item:intro.count.3}$, since it contains most of the ideas in the proof. We first show that $|X(\mathbb{Z} / p^m)|=|X(\mathbb{F}_p[t]/t^m)|$, for all but finitely many primes $p$. Next, we consider the function $F(p,m):=H_{X}(p^{m})=\frac{|X(\mathbb{Z} / p^m)|}{p^{m \cdot \dim X}}$. We show the following: \begin{enumerate}[(a)]
\item \label{item:F.bdd} For every $p$, the function $m \mapsto F(p,m)$ is bounded. This follows from \cite{AA}.
\item \label{item:estimate.bdd.m} Condition \eqref{item:intro.count.3} holds for bounded $m$. To show this, we analyze the jet schemes $\Jet_nX$, which are schemes that satisfy $\Jet_nX(\mathbb{F}_p)\cong X(\mathbb{F}_p[t]/t^n)$. We apply a theorem of Mustata claiming that, under our assumption, $\Jet_nX$ are irreducible, and the Lang--Weil bounds.
\item \label{item:F.formula} There is a formula for $F(p,m)$ involving simple expressions in $m$ and $p$, as well as the number of $\mathbb{F}_p$-points on finitely many varieties. This is a consequence of motivic integration.
\end{enumerate}
We deduce from \eqref{item:F.formula} and \eqref{item:F.bdd} (and the Lang--Weil bounds) that there are finitely many $m_1,\ldots,m_k$ such that, for all $m$, there is $i$ such that $|F(m,p)-F(m_i,p)|=O(p ^{-1})$. Applying \eqref{item:estimate.bdd.m}, we get the result.

\subsection{Counting representations}

For a topological group $\Gamma$, let $r_n(\Gamma)$ be the number of {isomorphism classes of irreducible, $n$-dimensional, complex, continuous representations of $\Gamma$}. The sequence $r_n(\Gamma)$ is called the representation growth sequence of $\Gamma$. In general, $r_n(\Gamma)$ can be infinite, but we will only consider groups for which $r_n(\Gamma)$ is finite for any $n$.

The study of the representation growth sequence was introduced at \cite{Jai} and \cite{LM}. For more recent results, we refer the reader to \cite{AA}, \cite{A}, \cite{AKOV}, \cite{LL}, \cite{LS}, and the references therein. 

Throughout the paper, fix an affine group scheme $G$ over $\mathbb{Z}$ whose generic fiber $G_\mathbb{Q}$ is $\mathbb{Q}$-simple, connected, and simply connected. If $k$ is a global field and $T$ is a finite set of places of $k$ containing all archimedean ones, we denote by $O_{k,T}$ the ring $T$-integers:
\[
O_{k,T}= \left\{ x\in k \mid \forall v\notin T \text{ we have } \| x\|_v \leq 1\right\} .
\]

We will study the representation growth of the group  $\Gamma=G(O_{k,T})$.  The main theorem of \cite{LM} implies that, if $\Gamma$ satisfies the congruence subgroup property (See Definition \ref{defn:CSP} below), then the sequence $r_n(\Gamma)$ is bounded by some polynomial in $n$. Examples of groups with the  congruence subgroup property are $G(O_{k,T})$ assuming the $\mathbb{Q}$-rank of $G_\mathbb{Q}$ is greater than one. In order to catch the rate of polynomial growth of $r_n(\Gamma)$, we use the following definition:

\begin{definition*} Let $\Gamma$ be a topological group, and assume that the sequence $r_n(\Gamma)$ is bounded by a polynomial. The representation zeta function of $\Gamma$ is the Dirichlet series
\begin{equation} \label{eq:defn.zeta}
\zeta_\Gamma(s)=\sum_{n=1}^ \infty r_n(\Gamma)n^{-s}.
\end{equation}
Denote the abscissa of convergence of $\zeta_\Gamma$ by $\alpha(\Gamma)$.
\end{definition*}

As before, we have:
\begin{equation} \label{eq:alpha.limsup}
\alpha(\Gamma)=\limsup_{n \rightarrow \infty} \frac{\log \left( r_1(\Gamma)+\cdots+r_n(\Gamma)\right)}{\log n}.
\end{equation}
If $\Gamma=G(O_{k,T})$ as above, the $\limsup$ in \eqref{eq:alpha.limsup} is actually a limit and is a rational number (see \cite{A}).

The abscissa $\alpha(\Gamma)$ is related to the singularities of the varieties parameterizing homomorphisms from surface groups to $G$:
\begin{definition*} Let $n \in \mathbb{Z}_{\geq 1}$ and let $\Sigma_n$ be the closed surface of genus $n$. The deformation variety of $\Sigma_n$ in $G$ is defined to be the variety
\[
\Def_{G,n}=\Hom(\pi_1(\Sigma_n),G)=\left\{ (g_1,h_1,\ldots,g_n,h_n)\in G \mid [g_1,h_1]\cdots[g_n,h_n]=1\right\}.
\]
\end{definition*}

In \cite[Theorem VIII]{AA}, we showed that $\Def_{G_\mathbb{Q},n}$ has rational singularities if $n \geq \mathcal{C}(G)$, where
\[
\mathcal{C}(G)=\left\{\begin{matrix} 12 & G \textrm{ is of type } A,B,D \\ 21 & G \textrm{ is of type } C  \\ \lceil \frac{3}{2}(\dim(\mathfrak{g})+1) \rceil&\quad \textrm{$G$ is exceptional and  $\mathfrak{g}$  is a simple factor  of $\Lie(G_{\overline{\mathbb{Q}}})$} \end{matrix} \right. .
\]

Note that ${\mathcal{C}}(G)$ is bounded by $374$ for any $G$.

%
%
%

The following is a generalization of Theorem \ref{thm:intro.bnd}:

\begin{romantheorem}[See \S\S\ref{ssec:pf.bound}] \label{thm:introduction.bound}  There is a finite set $S$ of prime numbers such that, for every global field  $k$ of characteristic not in $S$ and every finite set $T$  of places of $k$ containing all archimedean ones, if $G(O_{k,T})$ satisfies the congruence subgroup property and $\Def_{G_\mathbb{Q},n}$ has rational singularities, then $\alpha(G(O_{k,T})) \leq 2n-2$.
\end{romantheorem}


 As a result, we get the following dichotomy for the representation growth of an arithmetic group $\Gamma$ in characteristic zero: either $r_n(\Gamma)=o(n^{747})$, or $r_n(\Gamma)$ grows super-polynomially in $n$ (this happens if $\Gamma$ does not satisfy the Congruence Subgroup Property---see \cite{LM}).

We will deduce Theorem \ref{thm:introduction.bound} from the following adelic version, which is applicable also for low rank groups. In the following, if $A$ is a ring, we denote its pro-finite completion by $\widehat{A}$.

\begin{romantheorem}[See \S\S\ref{ssec:pf.main}] \label{thm:introduction.main} There is a finite set $S$ such that, for any global field $k$ of characteristic not in $S$, any finite set $T$ of places of $k$ containing all archimedean places, and any natural number $n$, if $\Def_{G_\mathbb{Q},n}$ has rational singularities, then $\alpha \left( G(\widehat{O_{k,T}})\right) \leq 2n-2$.
\end{romantheorem}

From Theorem \ref{thm:introduction.main} and Theorem \ref{thm:Frobenius} below, we deduce the following statement about finite groups like $G(\mathbb{Z}/N \mathbb{Z})$, generalizing \cite[Corollary XI]{AA}:

\begin{romancorollary} \label{cor:prob.fin} There is a finite set $S$ of primes such that, for any ring of integers $O_{k,T}$ of any global field $k$ of characteristic not in $S$, there is a constant $C$ such that the following holds. For any non-trivial ideal $I$ of $O_{k,T}$, any natural number $n \geq 3$ such that $\Def_{G_\mathbb{Q},n-1}$ has rational singularities, and any $g\in G(O_{k,T}/I)$,
\[
\mathrm{Prob}\left( [g_1,h_1] \cdots [g_{n},h_{n}]=g\right) < \frac{C}{|G(O_{k,T}/I)|},
\]
where $g_1,h_1,\ldots,g_n,h_n$ are uniformly distributed random elements of $G(O_{k,T}/I)$.
\end{romancorollary}

We deduce Theorem \ref{thm:introduction.main} from the  following one:
\begin{romantheorem}[See \S\S\ref{ssec:pf.zeta.minus.1}] \label{thm:zeta.minus.1} For every $n\geq 2$ such that $\Def_{G_\mathbb{Q},n}$ has rational singularities, there is a finite set $S$ of prime numbers and an integer $C$ such that, for any $p\notin S$ and any unramified extension $F$ of either $\Q_p$ or $\F_p((t))$, we have
$$\zeta_{G(O_F)}(2n-2)-1 \leq C |O_F/I_F|^{-1},$$
where $O_F$ is the ring of integers of $F$ and $I_F$ is its maximal ideal.

\end{romantheorem}

In order to prove Theorem  \ref{thm:zeta.minus.1}, we use the Frobenius formula (Theorem \ref{thm:Frobenius} below) that relates $\zeta_{G(O_F/I_F^m)}(2n)$ to the sizes of  $\Def_{G,n}(O_F/I_F^{m})$. These sizes, in turn, are related to the singularities of  $\Def_{G,n}$ by Theorem \ref{thm:intro.count} (or, more precisely, by Theorem \ref{thm:size.X}).

\begin{remark*} Consider the following statements: \begin{enumerate}
\item \label{cond:introduction.main.1} $\alpha({G(O)})<2n-2$.
\item \label{cond:introduction.main.2} $\alpha({G(O_v)})<2n-2$ for any valuation $v$ of $O$.
\item \label{cond:introduction.main.3} $\Def_{G,n}$ has rational singularities.
\item \label{cond:introduction.main.4} $\alpha({G(O)})\leq 2n-2$.
\end{enumerate}
Then $\eqref{cond:introduction.main.1} \implies \eqref{cond:introduction.main.2} \iff \eqref{cond:introduction.main.3} \implies \eqref{cond:introduction.main.4}$.

The implication $\eqref{cond:introduction.main.1} \implies \eqref{cond:introduction.main.2}$ follows easily from the observation that, using the map $G(O) \rightarrow G(O_v)$, we get that $r_n(G(O_v)) \leq r_n(G(O))$. The equivalence $\eqref{cond:introduction.main.2} \iff \eqref{cond:introduction.main.3}$ is proved in \cite{AA}. In this paper, we prove $\eqref{cond:introduction.main.3} \implies \eqref{cond:introduction.main.4}$.

\end{remark*}

\begin{remark*} The assumption that $G$ is defined over $\mathbb{Z}$ can be weakened to assuming that $G$ is defined over the ring of $T$-integers in a number field. Indeed, if $G \subset \GL_{O_{k,T}}$ is defined over $O_{k,T}$, let $\mathcal{G}$ be the Zariski closure of $G$ in $\GL_{O_{k}}$. The theorems above applied to the restriction of scalars $\Res_{O_k/ \mathbb{Z}} \mathcal{G}$ imply the theorems for $G$.
\end{remark*}

\subsection{Structure of the paper}

In \S \ref{sec:prel} we give the necessary preliminaries for the paper. In \S\S\ref{subsec:sing}-\S\S\ref{subsec:jet} we review the relevant algebraic geometry. In \S\S\ref{subsec:FRS} we review the results of \cite{AA}. In \S\S\ref{sebsec:FF} we recall a theorem of Frobenius on representations of finite groups. In \S\S\ref{subsec:def.int} we review relevant parts from the theory of motivic integration.

In \S\ref{sec:count} we prove (a generalization of) Theorem \ref{thm:intro.count}. In \S\S\ref{ssec:count.integral},\S\S\ref{ssec:LW}, and \S\S\ref{ssec:Irr.comp}, we review the tools we use for the proof. The proof itself is given in \S\S\ref{ssec:DL}, \S\S\ref{ssec:deprecise}, \S\S\ref{ssec:horz.to.RS}, and \S\S\ref{ssec:RS.to.rate.precise}.

In \S\ref{sec:global} we prove Theorems \ref{thm:igu}, \ref{thm:introduction.bound}, \ref{thm:introduction.main}, and \ref{thm:zeta.minus.1}. As mentioned above, these imply Theorem \ref{thm:intro.bnd} and Corollary \ref{cor:prob.fin}.

\subsection{Acknowledgements}
{
We thank Vladimir Hinich for useful discussions. Part of this paper was written during the program `Multiplicity Problems in Harmonic Analysis' held at the Hausdorff Institute (2012-2014). A.A. was partially supported by ISF grant 687/13, NSF grant DMS-1100943, and a Minerva foundation grant. N.A. was partially supported by NSF grants DMS-0901638 and DMS-1303205. Both of us were partially supported by BSF grant 2012247.
}

\section{Preliminaries} \label{sec:prel}

\subsection{Conventions} Throughout this article we use the following conventions:
\begin{itemize}
\item Let $X \rightarrow S$ and $Y \rightarrow S$ be $S$-schemes. We denote $X \times_{S} Y$ by $X_Y$. If $Y=\Spec R$ is affine, we write $X_R$ instead of $X_{\Spec R}$.
\item If $p$ is a prime number and $q=p^n$ is a power of $p$, we denote the unique degree-$n$ unramified extension of $\mathbb{Q}_p$ by $\mathbb{Q}_q$, its ring of integers by $\mathbb{Z}_q$, and the maximal ideal of $\mathbb{Z}_q$ by $\mathfrak{m}_q$.
\item For a set $S$ of prime numbers, let $$\mathcal{P}_S=\left\{ p^n \mid \textrm{$p$ is a prime number not in $S$ and $n\in \mathbb{Z}_{\geq 1}$} \right\}$$
be the set of sizes of finite fields of characteristics not in $S$.
\item If $X \rightarrow S$ is a smooth map of relative dimension $d$, let $\Omega_{X / S}^d$ denote the line bundle of relative top differential forms.
\end{itemize}

\subsection{Singularities} \label{subsec:sing}
In this section, we review the notions of resolution of singularities, rational singularities, and complete intersections. For more detailed overview, we refer the reader to \cite[Appendix B]{AA}.
\begin{definition} Let $X$ be an algebraic variety defined over a field $k$. A \emph{resolution of singularities} of $X$ is a proper map $p:\tilde X \to X$ such that $\tilde X$ is smooth, $p$ is birational, and the restriction of $p$ to $p ^{-1} (X^{sm})$ is an isomorphism.

\end{definition}
\begin{definition}[{cf. \cite[I \S 3, page 50-51]{KKMS}}] \label{defn:rational.singularities} Let $X$ be an algebraic variety defined over a field $k$ of characteristic $0$.
\begin{enumerate}
\item We say that $X$ has rational singularities if, for any (equivalently, for some) resolution of singularities $p:\tilde X \to X$, the natural morphism  $Rp_*(\mathcal{O}_{\tilde X}) \to \mathcal{O}_{X}$ is an isomorphism.
\item A (usually singular) point $x \in X(k)$ is a rational singularity if there is a Zariski neighborhood $U\subset X$ of $x$ that has rational singularities.
\end{enumerate}
\end{definition}

\begin{definition}[] \label{defn:lci} Let $X$ be a scheme of finite type over a ring $R$. \begin{enumerate}
\item $X$ is called a complete intersection if there is an affine and smooth map $Y \rightarrow \Spec R$, a closed embedding $X \rightarrow Y$, commuting with the structure maps, and regular functions $f_1,\ldots,f_c\in O_Y(Y)$ such that the ideal of $X$ in $Y$ is generated by the $f_i$ and each $f_i$ is not a zero divisor in $O_Y(Y) / (f_1,\ldots,f_{i-1})$.
\item $X$ is called a local complete intersection if there is an open cover $U_i$ of $X$ such that each $U_i$ is a complete intersection.
\end{enumerate}
\end{definition}

\subsection{Jet schemes and Mustata's theorem}\label{subsec:jet}
In this section we recall the definition of jet schemes and quote one of our main tools, Mustata's theorem (see \cite{Mus}), which relates rational singularities and irreducibility of jet schemes. We will use repeatedly the following simple lemma:

\begin{lemma} Suppose that $Z \rightarrow T \rightarrow S$ and $X \rightarrow S$ are morphisms of schemes. Then $\Hom_T(Z,X_T) \cong \Hom_S(Z,X)$.
\end{lemma}

We move on to define jet schemes.

\begin{notation} For a scheme $Y$, denote $Y^{[m]}=Y \times_{\Spec \mathbb{Z}} \Spec \mathbb{Z}[t]/t^{m+1}$.
\end{notation}
The projection $Y^{[m]} \rightarrow Y$ is finite and (locally) free. For every scheme $S$, the assignment $Y \mapsto Y^{[m]}$ gives rise to a functor between $(\Sch_S)$ and $(\Sch_{S^{[m]}})$.

\begin{notation} Let $X \rightarrow S$ be an affine morphism of finite type. Let $\mathscr{J}_m(X/S)$ be the restriction of scalars of $X^{[m]}$ along the map $S^{[m]} \rightarrow S$, i.e., the functor $\mathscr{J}_m(X/S):(\Sch_S) \rightarrow (\Set)$ given by
\[
\mathscr{J}_m(X/S)(Z)=\Hom_{S^{[m]}}(Z^{[m]},X^{[m]})\cong \Hom_S(Z^{[m]},X).
\]
\end{notation}

From \cite[\S 7.6]{NeronModels}, we get
\begin{thmdef} The functor $\mathscr{J}_m(X/S)$ is representable by a scheme of finite type over $S$. We call the representing scheme the $m$-th relative jet scheme of $X \rightarrow S$ and denote it by $\Jet_m(X/S)$.
\end{thmdef}

Let $X\rightarrow S$ be a morphism as above. For any morphisms $Z \rightarrow T \rightarrow S$ of schemes, we have a canonical bijections
\[
\Hom_T(Z,\Jet_m(X_T/T))\cong\Hom_T(Z^{[m]},X_T)\cong\Hom_S(Z^{[m]},X)\cong
\]
\[
\cong\Hom_S(Z,\Jet_m(X/S))\cong\Hom_T(Z,\Jet_m(X/S)_T),
\]
which is functorial in $Z$, and, therefore, defines an isomorphism $\Jet_m(X_T/T)\cong\Jet_m(X/S)_T$.

The following is Mustata's theorem:

\begin{theorem}[\cite{Mus}] Let $X$ be an irreducible local complete intersection variety defined over a field $k$ of characteristic $0$. Then the following are equivalent: \begin{enumerate}
\item $X$ has rational singularities.
\item All the jet schemes $\Jet_m(X/k)$ are irreducible.
\end{enumerate}

\end{theorem}
\subsection{Rational singularities and integration}\label{subsec:FRS}

In this section we review a result of \cite{AA}. Recall that a measure $m$ on a $p$-adic analytic manifold\footnote{A $p$-adic manifold is a Hausdorff space $X$ with a sheaf of functions that is locally isomorphic to the space $\mathbb{Z}_p^N$ together with the sheaf of functions that are locally given by convergent power series; see \cite{Se}.} is Schwartz if it is compactly supported and every point has a neighborhood $U$ and a diffeomorphism $\varphi :U \rightarrow \mathbb{Z}_p^n$ such that $\varphi_* m$ is a scalar multiple of the Haar measure.

\begin{theorem} \lbl{thm:push.forward.detailed} Let $\varphi : X \rightarrow Y$ be a map between smooth algebraic varieties defined over a non-archimedean local field $F$ of characteristic zero, and let $x \in X(F)$. Assume that $\varphi$ is flat at $x$, and $x$ is a rational singularity of $\varphi ^{-1} (\varphi(x))$. Then, there is a neighborhood $U \subset X(F)$ of $x$ such that, for any Schwartz measure $m$ on $U$, the measure $\varphi_*(m) $ has continuous density, i.e., can be written as a product of a Schwartz measure and a continuous function.
\end{theorem}


\subsection{Frobenius formula} \label{sebsec:FF}

We will use the following Theorem of Frobenius:
\begin{theorem}[Frobenius] \lbl{thm:Frobenius}  Let $\Gamma$ be a finite group, and let $n \geq 1$ be an integer. Then
\[
\left| \left\{ (x_1,y_1,\ldots,x_n,y_n)\in \Gamma ^{2n} \mid [x_1,y_1] \cdots [x_n,y_n]=1 \right\} \right| = | \Gamma |^{2n-1}\sum_{\pi \in \Irr(\Gamma)} \frac{1}{(\dim\, \pi)^{2n-2}},
\]
where $\Irr(\Gamma)$ denotes the set of isomorphism classes of irreducible representations of $\Gamma$.
\end{theorem}
\subsection{Definable Integrals}\label{subsec:def.int}

In this section, we recall the setting of definable $p$-adic integrals and state a uniformity result about $p$-adic integrals, which is a special case of \cite[Theorem 1.3]{HK}.

\subsubsection{The Denef--Pas language} Let $L_\emptyset$ be the first-order language with:

\begin{itemize}
\item Three sorts, denoted by $\VF$, $\RF$, and $\VG$, and called the valued field sort, the residue field sort, and the valuation group sort, respectively.
\item Five constants, $0_{\VF},1_{\VF} \in \VF$, $0_{\RF},1_{\RF}\in \RF$, and $0_{\VG},\infty_{\VG} \in \VG$.
\item Seven functions, $+_{\VF}:\VF \times \VF \rightarrow \VF$, $\cdot_{\VF} :\VF \times \VF \rightarrow \VF$, $+_{\RF}:\RF \times \RF \rightarrow \RF$, $\cdot_{\RF}:\RF \times \RF \rightarrow \RF$, $+_{\VG} :\VG \times \VG \rightarrow \VG$, $\val: \VF \rightarrow \VG$, and $\ac: \VF \smallsetminus \left\{ 0_{\VF} \right\} \rightarrow \RF$.
\item One binary relation, $<$, on $\VG$.
\end{itemize}

\subsubsection{Structures} Suppose that $F$ is a field with a non-archimedean valuation $v$. Denote $O=\left\{ x \in F \mid v(x) \geq 0 \right\}$ and $\mathfrak{m}=\left\{ x\in F \mid v(x)>0 \right\}$. Assume that the short exact sequence
\begin{equation} \label{eq:ses.valued.field}
1 \rightarrow O^\times / (1+ \mathfrak{m}) \rightarrow F^ \times / (1+\mathfrak{m}) \rightarrow F^ \times / O^ \times \rightarrow 1
\end{equation}
splits, and let $\sigma : F^\times / (1+ \mathfrak{m}) \rightarrow O^\times /(1+ \mathfrak{m})$ be such a splitting. An important example is the case where $F$ is a non-archimedean local field. In this case, any uniformizer gives a splitting of \eqref{eq:ses.valued.field}.

From this data, we construct a structure for $L_\emptyset$ as follows: the sort $\VF$ is interpreted as $F$, the sort $\RF$ is interpreted as the residue field of $F$, and the sort $\VG$ is interpreted as {$\mathbb{Z}\cup \left\{ \infty \right\}$}. The function $\val$ is interpreted as the valuation $v$ and the function $\ac$ is interpreted as the composition
\[
F ^ \times \rightarrow F ^\times / (1+ \mathfrak{m}) \stackrel{\sigma}{\rightarrow} O^ \times / (1+ \mathfrak{m}) =\RF(F) \smallsetminus \left\{ 0 \right\}.
\]
The interpretation of the constants, relation, and the rest of the functions is clear. From this point on, when we write a valued field, we mean a valued field together with a splitting as above, and we consider it as a structure of the Denef--Pas language.



\subsubsection{Quantifier-free definable functions}
Suppose that $\mathfrak{S}$ is a structure for $L$. If $\varphi=\varphi(x_1,\ldots,x_n,y_1,\ldots,y_m,z_1,\ldots,z_k)$ is a formula in $L$, where the variables $x_i$ are of $\VF$ sort, the variables $y_i$ are of $\RF$ sort, and the variables $z_i$ are of $\VG$ sort, we denote
\[
\varphi(\mathfrak{S})= \left\{ (a_i,b_i,c_i) \in \VF(\mathfrak{S})^n \times \RF(\mathfrak{S})^m \times \VG(\mathfrak{S})^k \mid \textrm{$\varphi(a_i,b_i,c_i)$ holds in $\mathfrak{S}$} \right\}.
\]

\begin{definition}
$ $
\begin{enumerate}
\item We say that two formulas $\phi(x)$ and $\psi(x)$ are equivalent if, for every Henselian valued field $F$, we have $\phi(F)=\psi(F)$. An equivalence class of formulas is called a definable set. If $X$ is a definable set and $F$ is a Henselian valued field, we write $X(F)=\phi(F)$, for any $\phi \in X$.
\item We say that a formula in $L$ is quantifier-free if there are no quantifiers in it. A definable set is called quantifier-free, if there is a quantifier-free formula representing it.
\item Suppose that $X,Y$ are definable sets on variables $x,y$ respectively ($x$ and $y$ are tuples of variables of the three sorts of $L$). A quantifier-free definable set $\Gamma$ on the tuple of variables $(x,y)$ is called a quantifier-free definable function if, for any Henselian valued field $F$, the set $\Gamma(F)$ is the graph of a function between $X(F)$ and $Y(F)$, which we also denote by $\Gamma(F)$.
\end{enumerate}
\end{definition}

\begin{example}
$ $

\begin{enumerate}
\item Let $X \subset \mathbb{A}_{\mathbb{Z}}^N$ be an affine scheme over $\Spec \mathbb{Z}$. Choose a generating set $\left\{ p_1,\ldots,p_M \right\} \subset \mathbb{Z}[x_1,\ldots,x_N]$ for the ideal of polynomials vanishing on $X$, and let $X_{\VF}$ be the equivalence class of the formula
\[
\left( p_1(x_1,\ldots,x_N)=0 \right) \wedge \cdots \left( p_M(x_1,\ldots,x_N)=0 \right),
\]
where $x_1,\ldots,x_N$ are of the $\VF$ sort. For any valued field $F$, we have $X_{\VF}(F)=X(F)$. Similarly, there is also a quantifier-free definable set, denoted by $X_{\RF}$, such that $X_{RF}(F)=X(k)$, for all Henselian valued fields $F$ with residue field $k$.
\item Let $f: X \rightarrow Y \rightarrow \Spec \mathbb{Z}$ be morphisms of schemes. Then there is a quantifier-free definable function $f_{\VF}:X_{\VF} \rightarrow Y_{\VF}$ such that, for every Henselian valued field $F$, the function $f_{\VF}(F): X_{\VF}(F) \rightarrow Y_{\VF}(F)$ coincides with the restriction of $f$ to $F$ points.
\item If $X$ is an algebraic variety and $p$ is a regular function on $X$, the formula $y=\val(p(x))$ gives rise to a quantifier-free definable function from $X_{\VF}$ to $\VG$.
\end{enumerate}
\end{example}

The following theorem follows from \cite[Theorem 1.3]{HK}:

\begin{theorem} \label{thm:integration.q.f.} Let $X \rightarrow \Spec \mathbb{Z}$ be an affine and smooth morphism of relative dimension $d$, let $\omega \in \Gamma(X,\Omega_{X/\Spec \mathbb{Z}}^d)$, and let $f_1,f_2: X_{\VF} \rightarrow \VG$ be quantifier-free definable functions. Then there are \begin{enumerate}
\item A constant $M$.
\item A finite set $S$ of prime numbers.
\item (Finite or infinite) arithmetic progressions $I_1,\ldots,I_M \subset \mathbb{Z}$.
\item Polynomials $g_1,\ldots,g_M\in \mathbb{Q}[x,y]$ such that $g_j(x,y)$ is positive on $I_j \times \mathbb{R}_{>1}$.
\item Reduced affine schemes $V_1,\ldots,V_M$ of finite type over $\mathbb{Z}$ which are smooth over $\Spec \mathbb{Z}[S ^{-1}]$ such that $(V_i)_\mathbb{Q}$ are non-empty and irreducible.
\item Rational numbers $\alpha_1,\ldots,\alpha_M,\beta_1,\ldots,\beta_M$.
\item Positive integers $a_{j,k}$ for $j=1,\ldots,M$ and $k=1,\ldots,M_j$.
\end{enumerate}
such that, if $F$ is a local field with residue field $\mathbb{F}_q$ of characteristic not in $S$, $m\in \mathbb{Z}$, and $| \omega |_F$ denotes the measure on $X(F)$ corresponding to $\omega$, then
\begin{equation} \label{eq:motivic.integration}
\int_{\left\{ x\in X(O_F) \mid f_1(x)=m\right\}} q^{-f_2(x)}| \omega |_F = \sum_{j=1}^M 1_{I_j}(m) |V_j(\mathbb{F}_q)| \frac{g_j(m,q) q^{\alpha_j m+ \beta_j}}{\prod_{k=1}^{M_j} \left( 1-q^{-a_{jk}} \right)},
\end{equation}
assuming the integral converges.
\end{theorem}

\section{Rational singularities and points over finite local rings} \label{sec:count}
\setcounter{theorem}{0}

In this section, we study the number of points of schemes over finite rings. For this, we introduce the following notation:
\begin{definition} Let $X$ be a scheme of finite type over $\mathbb{Z}$ and let $A$ be a finite ring. If $X_\mathbb{Q}$ is nonempty, define $h_X(A)=\frac{|X(A)|}{|A|^{\dim X_{\mathbb{Q}}}}$. If $X_\mathbb{Q}$ is empty, let $h_X(A)=0$.
\end{definition}
We will mostly consider the finite rings $\mathbb{Z}_q / \mathfrak{m}_q^m$ and $\mathbb{F}_q[t]/t^m$. The following proposition relates the counting functions for these rings.

\begin{proposition} \label{prop:Denef.Loeser} (see \S\S\ref{ssec:DL}) Let $X$ be a scheme of finite type over $\mathbb{Z}$. There is a finite set $S$ of prime numbers such that $h_X(\mathbb{Z}_q / \mathfrak{m}_q^m)=h_X(\mathbb{F}_q[t]/t^m)$, for any $q\in \mathcal{P}_S$ and any $m\in \mathbb{N}$.
\end{proposition}

The main result of this section is the following theorem that generalizes Theorem \ref{thm:intro.count}:

\begin{theorem} \label{thm:size.X} Let $X$ be a scheme of finite type over $\mathbb{Z}$ such that $X_{\mathbb{Q}}$ is equi-dimensional and a local complete intersection. The following are equivalent:
 \begin{enumerate}[(i)]
\item \label{item:lim.1.horz} For any $m$, $\lim_{p \rightarrow \infty} h_X(\mathbb{Z} / p ^m) = \lim_{p \rightarrow \infty} h_X(\mathbb{F}_p[t] / t^m) = 1$.
\item \label{item:rate} There is a finite set $S$ of prime numbers {and a constant C} such that $|h_X(\mathbb{Z}_q / \mathfrak{m}_q ^m)-1|=|h_X(\mathbb{F}_q[t] / t ^m)-1|<C q^{-1/2}$, for any $q\in \mathcal{P}_S$ and $m \in \mathbb{N}$.
\item \label{item:rate.precise} $X_{\overline{\mathbb{Q}}}$ is irreducible and there is a finite set $S$ of prime numbers {and a constant C} such that $|h_X(\mathbb{Z}_q / \mathfrak{m}_q ^m)-h_X(\mathbb{F}_q)|=|h_X(\mathbb{F}_q[t] / t ^m)-h_X(\mathbb{F}_q)|<C q^{-1}$, for any $q\in \mathcal{P}_S$ and $m \in \mathbb{N}$.
\item \label{item:RS} $X_{\overline{\mathbb{Q}}}$ is reduced, irreducible, and has rational singularities.
\end{enumerate}
\end{theorem}

\begin{remark} \label{rem:l.c.i.h.bdd}
The proof of the Theorem \ref{thm:size.X} proves also that, if $X$ itself is local complete intersection, then \eqref{item:RS} implies:
\begin{enumerate}[(i)]
\setcounter{enumi}{4}
\item \label{item:bdd.vert} $X_{\overline{\mathbb{Q}}}$ is irreducible and for any prime power $q$, the sequence $n \mapsto h_X(\mathbb{Z}_q / \mathfrak{m}_q^n)$ is bounded.
\end{enumerate}
We conjecture that this implication is true without this additional assumption. Moreover, we conjecture that \eqref{item:lim.1.horz}-\eqref{item:bdd.vert} are also equivalent to
\begin{enumerate}[(i)]
\setcounter{enumi}{5}
\item \label{item:bdd.vert.weak} $X_{\overline{\mathbb{Q}}}$ is irreducible and there is a finite set $S$ of prime numbers such that, for any prime $p$ not in $S$, the sequence $n \mapsto h_X(\mathbb{Z} / p^n)$ is bounded.
\end{enumerate}
\end{remark}


\begin{remark} One can weaken Condition \eqref{item:rate} by replacing the set $\mathcal{P}_S$ by any sequence $p_n^{m_n}$ of prime powers such that the $p_n$ are distinct and any integral polynomial splits over some finite field of size $p_n^{m_m}$. For example, one can take the sequence of primes that are congruent to 1 modulo $n$.
\end{remark}

\subsection{Counting points and integrals} \label{ssec:count.integral}

\begin{proposition}\label{cor:mu} Let $F$ be a local field, $O$ its ring of integers, and $\mathfrak{m}$ its maximal ideal. Suppose that $Y \rightarrow \Spec O$ is a smooth map of schemes of (pure) relative dimension $d$. For any $m$, denote the map $Y(O) \rightarrow Y(O / \mathfrak{m} ^m)$ by $r_m$. There is a unique Schwartz measure $\mu$ on $Y(O)$ such that, for every $m$ and $y_0\in Y(O / \mathfrak{m} ^m)$,
\[
\mu(r_m ^{-1}(y_0))=|O / \mathfrak{m} |^{-dm}
\]
\end{proposition}

The Proposition follows from the following standard lemma:

\begin{lemma} \label{lem:Weil.measure} Let $F,O,\mathfrak{m},Y,d$, and $r_m$ be as in Proposition \ref{cor:mu}. Assume that there is an invertible section $\omega$ of $\Omega_{Y/\Spec O}^d$. Then, for every $y_0\in Y(O / \mathfrak{m} ^m)$,
\[
\int_{r_m ^{-1}(y_0)} | \omega |=|O/ \mathfrak{m} |^{-dm}
\]
\end{lemma}

\begin{corollary} \label{cor:motivic.affine} Let $Z$ be a scheme of finite type over $\mathbb{Z}$ and assume that $Z_\mathbb{Q}$ is affine. Then there are $M,S,I_j,g_j,V_j,\alpha_j,\beta_j,M_j,a_{j,k}$ as in Theorem \ref{thm:integration.q.f.} such that, for any $q\in \mathcal{P}_{S}$ and any $m \geq 0$,
\begin{equation} \label{eq:h.motivic}
h_Z(\mathbb{Z}_q / \mathfrak{m}_q^m)=h_Z(\mathbb{F}_q[t]/t^m)=\sum_{j=1}^M 1_{I_j}(m) |V_j(\mathbb{F}_q)| \frac{g_j(m,q) q^{\alpha_j m+ \beta_j}}{\prod_{k=1}^{M_j} \left( 1-q^{-a_{jk}} \right)}.
\end{equation}
\end{corollary}

\begin{proof} If $Z_\mathbb{Q}$ is empty, the claim is clear. Otherwise, let $S_0$ be a finite set of primes and let $f_1\dots f_l \in S_0^{-1}\Z[x_1,\dots,x_N]$ be polynomials such that  $Z_{S_0^{-1}\Z}$ is the scheme defined by $f_1=\dots=f_l=0$ in $\A^N_{S_0^{-1}\Z}.$ Let $F$ be a local field of characteristic not in $S$, $O$ be its ring of integers and $\mathfrak m$ be its maximal ideal. By Lemma \ref{lem:Weil.measure} we have:
\begin{multline*}
|Z(O/\mathfrak m^m)|=|O/\mathfrak m|^{Nm}\int_{\{x\in O^N \mid \min_j(\val(f_j(x))) \geq m\}}|d{x_1} \wedge \dots \wedge d{x_N}|=\\=|O/\mathfrak m|^{Nm} \sum_{n\geq m}\int_{\{x\in O^N \mid \min_j(\val(f_j(x)))=n\}}|d{x_1} \wedge \dots \wedge d{x_N}|.
\end{multline*}
The claim now follows from Theorem \ref{thm:integration.q.f.}.
\end{proof}

\subsection{Proof of Proposition \ref{prop:Denef.Loeser}} \label{ssec:DL}

We will use the following simple lemma.
\begin{lemma} \label{lem:simple.h}
Let $X=U_1\cup U_2$ be an open cover of a scheme. Then, for any finite ring $A$, we have
\begin{itemize}
\item $|X(A)|=|U_1(A)|+|U_2(A)|- |U_1\cap U_2(A)|.$
\item $|X(A)| \geq |U_1(A)|.$
\end{itemize}
\end{lemma}

We also need the following well known fact:
\begin{lemma} (cf. \cite[Ex. II.4.3]{Har}) \label{lem:affine.intersection}
Let $X$ be a separated scheme, and let $U_1,U_2 \subset X$ be open affine subschemes. Then $U_1\cap U_2$ is affine.
\end{lemma}

\begin{proof}[Proof of Proposition \ref{prop:Denef.Loeser}] We will show that there is a set $S$ of prime numbers such that $|X(\mathbb{Z}_q/ \mathfrak{m}_q^m)|=|X(\mathbb{F}_q[t]/t^m)|$, for all $q\in \mathcal{P}_S$ and all $m$. We prove this claim in the following cases: \begin{enumerate}[{Case} 1]
\item $X$ is affine. \\
The claim follows from Corollary \ref{cor:motivic.affine}.

\item $X$ is separated.\\
The proof is by induction on the minimal size of an affine cover of $X$. The base is the affine case dicussed above. For the transition, let $X=\bigcup_1^n U_i$ be an affine cover and assume that we proved the claim for schemes that can be covered by less then $n$ affine open subschemes. Let $U= \bigcup_2^n U_i$. By induction and Lemma \ref{lem:affine.intersection}, we know the statement for $U_1, U, U\cap U_1$. Thus Lemma \ref{lem:simple.h} implies the assertion.
\item The general case.\\
The proof is by induction on the minimal size of an affine cover of $X$ by separated schemes. The base is the previous case. The transition is analogous to the proof in the previous case.
\end{enumerate}

\end{proof}


\subsection{Lang--Weil estimates} \label{ssec:LW}

We start with the following notation.

\begin{notation} Suppose that $X$ is a scheme of finite type over $\Spec \mathbb{Z}$ and that $k$ is a field. We denote the number of irreducible components of $X _{\overline{k}}$ that are defined over $k$ and have dimension $\dim X_k$ by $c_X(k)$. If $q$ is a prime power, we also write $c_X(q)$ instead of $c_X(\mathbb{F}_q)$.
\end{notation}

\begin{theorem}[Lang--Weil (cf. \cite{LW})] Let $X$ be a scheme of finite type over $\Spec \mathbb{Z}$. Let $d=\dim X_\mathbb{Q}$. There is a finite set $S$ of prime numbers and a constant $C$ such that, for any $q\in \mathcal{P}_S$,
\[
\left| \frac{|X(\mathbb{F}_q)|}{q^d}-c_X(q)\right| < C q^{-\frac{1}{2}}.
\]
\end{theorem}

\subsection{Irreducible components modulo $p$} \label{ssec:Irr.comp}

In this subsection, we prove the following

\begin{proposition} \label{prop:Che} Suppose that $X$ is a scheme of finite type over $\Spec \mathbb{Z}$. Assume that $X_{\mathbb{Q}}$ is irreducible. Then \begin{enumerate}
\item\label{prop:Che:1} For almost all prime numbers $p$, we have $c_X(\overline{\mathbb{F}_p})=c_X(\overline{\mathbb{Q}})$.
\item\label{prop:Che:2} There are infinitely many prime numbers $p$ such that $c_X(p)=c_X(\overline{\mathbb{Q}})$.
\end{enumerate}
\end{proposition}

In order to prove the proposition, we recall the following notions and facts:

\begin{definition} Let $S$ be an irreducible scheme with generic point $\eta$, let $s$ be a closed point of $S$, and let $ \varphi :X \rightarrow S$ be a morphism of finite type. Define a relation $\mathfrak{R}_{S,s}$ between the set of irreducible components of $X_\eta :=\varphi ^{-1} (\eta)$ and the set of irreducible components of $X_s:=\varphi ^{-1} (s)$ by $(Y_1,Y_2)\in \mathfrak{R}_{S,s}$ if $Y_2 \subset \overline{Y_1}$.
\end{definition}

\begin{theorem} \label{thm:EGA.9.7.8} (cf. \cite[Proposition 9.7.8]{EGA4}) Let $S$ be an irreducible scheme with generic point $\eta$ and let $ \varphi :X \rightarrow S$ be a morphism of finite type. Assume that all irreducible components of $X_\eta$ are absolutely irreducible. There is an open set $U \subset S$ such that, for any closed point $s\in U$, \begin{enumerate}
\item All irreducible components of $X_s$ are absolutely irreducible.
\item The relation $\mathfrak{R}_{\eta,s}$ is the graph of a bijection.
\end{enumerate}
\end{theorem}

Let $X \rightarrow \Spec \mathbb{Z}$ be a scheme of finite type. Let $L/ \mathbb{Q}$ be a Galois extension such that all irreducible components of $X \times \Spec L$ are absolutely irreducible, and let $O$ be the ring of integers of $L$. The group $\Gal(L/ \mathbb{Q})$ acts on the set of irreducible components of $X \times \Spec L$. For any prime ideal $\mathfrak{p}$ of $O$ with residue field $k_\mathfrak{p}$, consider the decomposition group $D_\mathfrak{p}=\left\{ \sigma \in \Gal(L/ \mathbb{Q}) \mid \sigma(\mathfrak{p})=\mathfrak{p} \right\}$, and the homomorphism $\Phi_\mathfrak{p} :D_\mathfrak{p} \rightarrow \Gal(k_\mathfrak{p} / \mathbb{F}_p)$. The group $\Gal(k_\mathfrak{p}/ \mathbb{F}_p)$ acts on the set of irreducible components of $X_{k_\mathfrak{p}}$. The following is obvious

\begin{proposition} \label{prop:inv.rel} Under the assumptions above, for any $(Y_1,Y_2)\in \mathfrak{R}_{(0), \mathfrak{p}}$ and $\sigma \in D_\mathfrak{p}$, we have $(\sigma \cdot Y_1, \Phi_\mathfrak{p}(\sigma) \cdot Y_2)\in \mathfrak{R}_{(0), \mathfrak{p}}$.
\end{proposition}

Proposition \ref{prop:inv.rel}, Theorem \ref{thm:EGA.9.7.8}, and Chebotarev's density theorem imply Proposition \ref{prop:Che}.

The Lang--Weil bounds and Proposition \ref{prop:Che}\eqref{prop:Che:1} imply:

\begin{corollary} \label{cor:h.open.dense} Let $X$ be a scheme of finite type over $\Spec \mathbb{Z}$, and let $U \subset X$ be an open dense subscheme. There is finite set $S$ of prime numbers and a constant $C$ such that $h_X(\mathbb{F}_q)-h_U(\mathbb{F}_q)< C q ^{-1}$ for any $q\in \mathcal{P}_S$.
\end{corollary}

\subsection{Proof of the implication $\eqref{item:rate.precise} \implies \eqref{item:rate}$} \label{ssec:deprecise}
It is enough to show that there is a finite set $S'$ of prime numbers and a constant $C'$ such that $|h_X(\mathbb{F}_q)-1|<C'q^{-1/2}$ for all $q\in \mathcal{P}_{S'}$. By the assumption and Proposition \ref{prop:Che}\eqref{prop:Che:1}, there is a finite set $T$ of prime numbers such that $c_X(q)=1$ for all $q\in \mathcal{P}_T$. The Lang-Weil estimates imply the result.

\subsection{Proof of the implication $\eqref{item:lim.1.horz} \implies \eqref{item:RS}$} \label{ssec:horz.to.RS}

Let $m\in \mathbb{N}$ and $X^{(m)}=\Jet_m(X/\Spec \mathbb{Z})$. There is a finite set $S$ of prime numbers such that $X_{\mathbb{F}_p}$ is a local complete intersection and $\dim X^{(m)}_{\mathbb{F}_p} = \dim X^{(m)}_\mathbb{Q}$ for all $p\notin S$. By Proposition \ref{prop:Che}\eqref{prop:Che:1}, we can enlarge $S$ and assume, in addition, that $c_{X^{(m)}}(\overline{\mathbb{F}_p})=c_{X^{(m)}}(\overline{\mathbb{Q}})$ for all $p\notin S$. By the Lang--Weil estimates, we get that
\[
 \left| \frac{h_X(\mathbb{F}_p[t] / t^m)}{p^{\dim X^{(m)}_{\mathbb{F}_q}-m\dim X_\mathbb{Q} }}-c_{X^{(m)}}(p) \right| =\left| \frac{|X^{(m)}(\mathbb{F}_p)|}{p^{\dim X_{\mathbb{F}_p}^{(m)}}} -c_{X^{(m)}}(p)\right| \underset{p \rightarrow \infty}{\longrightarrow} 0,
\]
By Proposition \ref{prop:Che}\eqref{prop:Che:2}, there are infinitely many primes $p$ such that $c_{X^{(m)}}(p)=c_{X^{(m)}}(\overline{\mathbb{Q}})$. Taking the limit over those and using the assumption \eqref{item:lim.1.horz}, we get that $p^{m\dim X_\mathbb{Q} - \dim X^{(m)}_{\mathbb{F}_q}} \rightarrow c_{X^{(m)}}(\overline{\mathbb{Q}})$, which implies that $\dim X^{(m)}_\mathbb{Q} = m\dim X_\mathbb{Q}$ and that $c_{X^{(m)}}(\overline{\mathbb{Q}})=1$. If $X_\mathbb{Q}$ were not reduced, then $\dim X^{(2)}_\mathbb{Q}=\dim TX_\mathbb{Q}$ would be larger than $2\dim X_\mathbb{Q}$, a contradiction. Hence, $X_\mathbb{Q}$ is reduced. Since $X$ is a local complete intersection, each irreducible component of $X^{(m)}_\mathbb{Q}$ has dimension at least $m \cdot \dim X_\mathbb{Q}$. Thus, $X^{(m)}_\mathbb{Q}$ is absolutely irreducible. By Mustata's theorem, $X$ has rational singularities.

\subsection{The implication $\eqref{item:RS} \implies \eqref{item:rate.precise}$} \label{ssec:RS.to.rate.precise}
If $X_\mathbb{Q}$ is empty, the claim is trivial. Hence, in the following, we will assume that $X_\mathbb{Q}$ is non-empty.

\subsubsection{The case of fixed $m$}
\begin{proposition} \label{prop:weak.rate} Assume that $X$ satisfies condition \eqref{item:RS}. Then, for any $m$, there is a finite set $S(m)$ of prime numbers and a constant $C(m)$ such that $|h_X(\mathbb{F}_q[t]/t^m)-h_X(\mathbb{F}_q)|<C(m)q ^{-1}$ for all $q\in \mathcal{P}_{S(m)}$.
\end{proposition}

\begin{proof} Let $U \subset X$ be the smooth locus of the structure map $X \rightarrow \Spec \mathbb{Z}$. Since $X_{\mathbb{Q}}$ is reduced, $U_{\mathbb{Q}}$ is dense in $X_{\mathbb{Q}}$. It follows that there is a finite set $S_1$ of prime numbers such that $U_{{\mathbb{F}_p}}$ is a dense subvariety of $X_{{\mathbb{F}_p}}$, for all $p\notin S_1$. By enlarging $S_1$ if needed, we can assume that $\dim X_{\mathbb{F}_p}=\dim X_{\mathbb{Q}}$ for all $p\notin S_1$.

Denote $X^{(m)}=\Jet_m(X/\Spec \mathbb{Z})$ and $U^{(m)}=\Jet_m(U/\Spec \mathbb{Z})$. For any $p\notin S_1$ and any $n$, the natural projection $U^{(n+1)}_{\mathbb{F}_p} \rightarrow U^{(n)}_{\mathbb{F}_p}$ is an $\mathbb{A} ^{\dim U_\mathbb{Q}}$-bundle. In particular, $|U^{(m)}(\mathbb{F}_q)|=|U(\mathbb{F}_q)| \cdot q^{(m-1)\dim U_{\mathbb{F}_q}}$. Hence,
\begin{equation} \label{eq:weak.rate.3}
h_U(\mathbb{F}_q[t]/t^m)=h_U(\mathbb{F}_q).
\end{equation}

By Corollary \ref{cor:h.open.dense}, there is a finite set $S_2 \supset S_1$ of prime numbers and a constant $C_2$ such that
\begin{equation} \label{eq:weak.rate.1}
|h_X(\mathbb{F}_q)-h_U(\mathbb{F}_q)|<C_2 q ^{-1} \quad q\in \mathcal{P}_{S_2}.
\end{equation}

Since $X_{\mathbb{Q}}$ has rational singularities, Mustata's theorem implies that $U^{(m)}_{\mathbb{Q}}$ is dense in $X^{(m)}_{\mathbb{Q}}$. There is a finite set $S_3(m)$ of prime numbers such that $U^{(m)}_{\mathbb{F}_p}$ is dense in $X^{(m)}_{\mathbb{F}_p}$, for all $p\notin S_3(m)$. Applying Corollary \ref{cor:h.open.dense}, we get that there is a finite set $S_4(m) \supset S_3(m)$ of prime numbers and a constant $C_4$ such that
\begin{equation} \label{eq:weak.rate.2}
|h_X(\mathbb{F}_q[t]/t^m)-h_U(\mathbb{F}_q[t]/t^m)|=|h_{X^{(m)}}(\mathbb{F}_q)-h_{U^{(m)}}(\mathbb{F}_q)|<C_4 q ^{-1} \quad q\in \mathcal{P}_{S_4(m)}.
\end{equation}

The claim now follows from \eqref{eq:weak.rate.3}, \eqref{eq:weak.rate.1}, and \eqref{eq:weak.rate.2}.

\end{proof}

%

\subsubsection{$\eqref{item:RS} \implies \eqref{item:bdd.vert.weak}$ in the case case where $X_\mathbb{Q}$ is a complete intersection} \label{sssec:RS.bdd.vert}
Since $X_\mathbb{Q}$ is a complete intersection, there is a smooth algebraic variety $\overline{Y}$ over $\mathbb{Q}$, a closed embedding $X_\mathbb{Q} \rightarrow \overline{Y}$, and a regular map $\overline{\varphi} : \overline{Y} \rightarrow \mathbb{A} ^{N}_\mathbb{Q}$, which is flat above 0, such that $X_\mathbb{Q}\cong \overline{\varphi} ^{-1} (0)$. Choose a $\mathbb{Z}$-scheme $Y$ of finite type and a regular map $\varphi : Y \rightarrow \mathbb{A}^N_\mathbb{Z}$ such that $Y_\mathbb{Q} = \overline{Y}$ and $\varphi_\mathbb{Q} = \overline{\varphi}$. There is a finite set $S$ of primes such that $Y_{\mathbb{Z}[S ^{-1}]}$ is smooth over $\Spec \mathbb{Z}[S ^{-1}]$, the restriction $\varphi_{\mathbb{Z}[S ^{-1}]}$ is flat above 0, and the isomorphism $X_\mathbb{Q} \cong \overline{\varphi} ^{-1} (0)$ extends to an isomorphism $X_{\mathbb{Z}[S ^{-1}]} \cong \varphi_{\mathbb{Z}[S ^{-1}]}^{-1} (0)$. It is enough to show, for any $p\notin S$, that $h_{\varphi_{\mathbb{Z}[S ^{-1}]}^{-1} (0)}(\mathbb{Z} /p^m)$ is bounded uniformly in $m$.


Fix $p\notin S$. Let $\mu$ be the measure on $Y(\mathbb{Z}_p)$ given by Corollary \ref{cor:mu}.  We have
\[
\mu \left( \varphi ^{-1} (p^m\Z_p^N) \right)  = \mu \left( r_m ^{-1} (\varphi_{\mathbb{Z} / p^m} ^{-1} (0)(\mathbb{Z} / p^m)) \right)  = \sum_{x \in \varphi_{\mathbb{Z} / p^m} ^{-1} (0)(\mathbb{Z} / p^m)} \mu \left( r_m ^{-1} (x) \right)  = p^{-m\dim Y_\Q}|\varphi_{\mathbb{Z} / p^m} ^{-1} (0)(\mathbb{Z} / p^m)|.
\]
By Theorem \ref{thm:push.forward.detailed}, there is a constant $C$ such that, for any $m \in \mathbb N,$
\[
\mu \left( \varphi ^{-1} (p^m\Z_p^N) \right)=\varphi_* \mu \left( p^m\Z_p^N \right) <\frac{C}{p^{mN}}.
\]
we get that
\[
|X(\mathbb{Z} / p^m)|=
|\varphi_{\mathbb{Z} / p^m} ^{-1} (0)(\mathbb{Z} / p^m)| < C p^{\dim Y_\Q-N}.
\]
Since $\dim Y_\Q-N=\dim X_\Q$, this implies the estimate we want.

\subsubsection{$\eqref{item:RS} \implies \eqref{item:rate.precise}$ in the case where $X_\mathbb{Q}$ is a complete intersection} \label{sssec:RS.rate.ci}

Assume now that $X$ is a complete intersection and satisfies \eqref{item:RS}. Applying Corollary \ref{cor:motivic.affine}, we get $M,S,I_j,g_j,V_j,\alpha_j,\beta_j,M_j,a_{j,k}$ such that
\begin{equation}\label{eq:RS.rate.ci}
h_X(\mathbb{Z}_q / \mathfrak{m}_q^m)=h_X(\mathbb{F}_q[t]/t^m)=\sum_{j=1}^M 1_{I_j}(m) |V_j(\mathbb{F}_q)| \frac{g_j(m,q) q^{\alpha_j m+ \beta_j}}{\prod_{k=1}^{M_j} \left( 1-q^{-a_{jk}} \right)}
\end{equation}
for all $q\in \mathcal{P}_S$ and $m$.

Define an equivalence relation on $\mathbb{N}$ by $m_1 \sim m_2$ iff  $ \forall j$ we have $m_1 \in I_j \Leftrightarrow m_2 \in I_j$.
We will show that there is a constant $C$
such that, for any $m_1 \sim  m_2$, we have
\begin{equation} \label{eq:boolean.algebra}
|h_X(\mathbb{Z}_q / \mathfrak{m}_q^{m_1})-h_X(\mathbb{Z}_q / \mathfrak{m}_q^{m_2})| < C q ^{-1}.
\end{equation}

Property \eqref{item:rate.precise} for $X$ will follow by applying Proposition \ref{prop:weak.rate} to representatives of the finitely many equivalence classes of $\sim$.

If the equivalence class of $m_1$ and $m_2$ is finite then
\eqref{eq:boolean.algebra} follows from Proposition \ref{prop:weak.rate}. Thus,  we can assume that all $I_j$ are infinite.

%

\begin{lemma} \label{lem:Pmq} Let $V$ be a scheme of finite type over $\Spec \mathbb{Z}$ whose generic fiber is non-empty, let $g(x,y)$ be a real polynomial, let $\alpha,\beta\in \mathbb{R}$, and let $a_1,\ldots,a_n$ be negative integers, 
Set $$P(m,q)=|V(\mathbb{F}_q)| \cdot \frac{g(m,q) q^{\alpha m+\beta}}{\prod_1^n \left( 1-q^{a_k} \right) }.$$  Then the following claims hold: \begin{enumerate}
\item Assume that either \begin{enumerate}
\item $\alpha >0$ or
\item $\alpha=0$ and $g$ depends on $m$ non-trivially (i.e. $\frac{\partial g}{\partial m} \neq 0$).
\end{enumerate}
Then $\lim_{m \rightarrow \infty} P(m,p)=\pm \infty$ for infinitely many primes $p.$
\item If $\alpha<0$ then there is a constant $D$ such that $|P(m,q)|<q ^{-1}$ for $m>D$ and any prime power $q$.
\end{enumerate}
\end{lemma}

\begin{proof} $ $
\begin{enumerate}
\item Under the assumptions, the limit $\lim_{m \rightarrow \infty} P(m,p)$ will be infinity if $V(\mathbb{F}_p)\neq \emptyset$. By the Lang--Weil bounds and Proposition \ref{prop:Che}\eqref{prop:Che:2}, this happens for infinitely many primes.
\item Suppose that the degree of $g$ is $d$. Then there is a constant $M$ such that $|g(m,q)| \leq Mm^dq^d$. In addition, There is a constant $e$ such that $|V(\mathbb{F}_q)|<eq^{e}$ for all prime powers $q$. Hence, there is a constant $C$ such that
\[
|P(m,q)| \leq Cm^dq^{\alpha m+\beta+d+e}.
\]
This implies the assertion.
\end{enumerate}
\end{proof}

%
%
%
By \S\S\ref{sssec:RS.bdd.vert}, after enlarging $S$, we have that
\[
\sup_{m} h_X(\mathbb{F}_p[t]/t^m)=\sup_{m}h_X(\mathbb{Z}/p^{m})< \infty,
\]
for any prime $p\notin S$. Since each summand in the right hand side of \eqref{eq:RS.rate.ci} is non-negative, we get that, for each $j$ and  $p\notin S$,
\begin{equation} \label{eq.lim.m}
\sup_{m} 1_{I_j}(m) |V_j(\mathbb{F}_q)| \frac{g_j(m,q) q^{\alpha_j m+ \beta_j}}{\prod \left( 1-q^{-a_{jk}} \right)} < \infty.
\end{equation}

For each $j$, Lemma \ref{lem:Pmq} implies that either $\alpha_j<0$ or $\alpha_j=0$ and $g_j$ is independent of $m$. Moreover, there is $D$ such that

\begin{equation}\label{eq:q.inv}
|V_j(\mathbb{F}_q)| \frac{g_j(m,q) q^{\alpha_j m+ \beta_j}}{\prod \left( 1-q^{-a_{jk}} \right)}< q ^{-1},
\end{equation}
 if $ \alpha_j <0$ and $m>D .$
 Let $J_1=\left\{ j \mid \alpha_j <0 \right\}$. 

By \eqref{eq:q.inv} and \eqref{eq:RS.rate.ci} we get that, for any $m>D$
 and $q\in \mathcal P_S,$ \[
\left| h_X(\mathbb{F}_q[t]/t^m)-\sum_{j\notin J_1} 1_{I_j}(m) |V_j(\mathbb{F}_q)| \frac{g_j(m,q) q^{\alpha_j m+ \beta_j}}{\prod \left( 1-q^{-a_{jk}} \right)}\right| <|J_1| q ^{-1}.
\]

Note that the sum on the left hand side depends only on $q$ and the equivalence class of $m$.
Thus, if $m_1,m_2>D$,  $m_1 \sim m_2$, and $q\in \mathcal P_S$, then
$$\left| h_X(\mathbb{F}_q[t]/t^{m_1})- h_X(\mathbb{F}_q[t]/t^{m_2})\right| <2|J_1| q ^{-1}.
 $$
 This, together with proposition \ref{prop:weak.rate},  implies  \eqref{eq:boolean.algebra}.

%
%
\subsubsection{$\eqref{item:RS} \implies \eqref{item:rate.precise}$ in general}
The proof is by induction on the size of the minimal open cover of $X_{\mathbb Q}$ by complete intersection varieties. The base is \S\S\ref{sssec:RS.rate.ci}. For the transition, let $X_{\mathbb Q}= \bigcup_1^n X'_i$ be a minimal open cover of $X_\mathbb{Q}$ by complete intersection varieties and assume that we know the assertion for any $Y$ such that $Y_{\mathbb Q}$ can be covered by less then $n$ complete intersection varieties. We can chose complete intersection schemes $X_i$ and a finite set of primes $S_0$ such that $X_{S_0^{-1}\mathbb Z}= \bigcup_1^n X_i$ is an open cover.

Let $U =\bigcup_2^n X_i$ and let $V \subset U \cap X_1$ be a complete intersection, open, and non-empty subscheme. By induction, we know that there is a finite set $S_{1}\supset S_0$ of prime numbers and a constant $C_1$ such that, for any $q\in \mathcal{P}_{S_{1}}$ and $m \in \mathbb{N}$, we have:
\begin{itemize}
\item $|h_{X_1}(\mathbb{Z}_q / \mathfrak{m}_q ^m)-h_{X_1}(\mathbb{F}_q)|<C_{1} q^{-1}$.
\item $|h_{U}(\mathbb{Z}_q / \mathfrak{m}_q ^m)-h_{U}(\mathbb{F}_q)|<C_{1} q^{-1}$.
\item $|h_{V}(\mathbb{Z}_q / \mathfrak{m}_q ^m)-h_{V}(\mathbb{F}_q)|<C_{1} q^{-1}$.
\end{itemize}

By Lemma \ref{lem:simple.h},
\begin{multline*}
h_{X}(\mathbb{Z}_q / \mathfrak{m}_q ^m)-h_{X}(\mathbb{F}_q) =\\= \left( h_{X_1}(\mathbb{Z}_q / \mathfrak{m}_q ^m)-h_{X_1}(\mathbb{F}_q) \right) + \left( h_{U}(\mathbb{Z}_q / \mathfrak{m}_q ^m)-h_{U}(\mathbb{F}_q) \right) - \left( h_{X_1 \cap U}(\mathbb{Z}_q / \mathfrak{m}_q ^m)-h_{X_1 \cap U}(\mathbb{F}_q) \right),
\end{multline*}
and so it is enough to prove that $|h_{X_1 \cap U}(\mathbb{Z}_q / \mathfrak{m}_q ^m)-h_{X_1 \cap U}(\mathbb{F}_q)|=O(q ^{-1})$.

By Corollary \ref{cor:h.open.dense}, there exist a finite set $S\supset S_1$ of prime numbers and a constant $C_2$ such that, for any $q\in \mathcal{P}_{S}$ and $m \in \mathbb{N}$, we have:

\begin{itemize}
\item $|h_{X_1}(\mathbb{F}_q)-h_{X_1 \cap U}(\mathbb{F}_q)|<C_{2} q^{-1}$.
\item $|h_{X_1 \cap U}(\mathbb{F}_q)-h_V(\mathbb{F}_q)|<C_{2} q^{-1}$.
\end{itemize}

Now,
\[
h_{X_1 \cap U}(\mathbb{Z}_q / \mathfrak{m}_q ^m)-h_{X_1 \cap U}(\mathbb{F}_q) \leq h_{X_1}(\mathbb{Z}_q / \mathfrak{m}_q ^m)-h_{X_1}(\mathbb{F}_q)+h_{X_1}(\mathbb{F}_q)-h_{X_1 \cap U}(\mathbb{F}_q)< (C_{1}+C_{2}) q^{-1},
\]
Similarly,
\[
h_{X_1 \cap U}(\mathbb{Z}_q / \mathfrak{m}_q ^m)-h_{X_1 \cap U}(\mathbb{F}_q) \geq h_{V}(\mathbb{Z}_q / \mathfrak{m}_q ^m)-h_V(\mathbb{F}_q)+h_V(\mathbb{F}_q)-h_{X_1 \cap U}(\mathbb{F}_q)>-(C_{1}+C_{2}) q^{-1}.
\]

\section{Zeta functions} \label{sec:global}

\subsection{Representation zeta functions of compact $p$-adic groups and the proof of Theorem \ref{thm:zeta.minus.1}} \label{ssec:pf.zeta.minus.1}
By \cite[Corollary 4.1.4]{AA}, the deformation variety $\Def_{G,n}$ is absolutely irreducible; by the assumptions, it has rational singularities. By  Proposition \ref{prop:Denef.Loeser} and Theorem \ref{thm:size.X},  we can choose a finite set $S$ of primes and a positive number $C$ such that
\begin{itemize}
\item For any $q \in \mathcal{P}_S$ and any positive integer $m$,
\begin{enumerate}[(a)]
\item\label{it:asum.1} $h_{\Def_{G,n}}(\Z_q/\mathfrak{m}_q^m)=h_{\Def_{G,2n}}(\F_q[t]/t^m).$
\item\label{it:asum.2} $h_{G}(\Z_q/\mathfrak{m}_q^m)=h_{G}(\F_q[[t]]/t^m).$
\item\label{it:asum.1.5} $|h_{\Def_{G,n}}(\mathbb{F}_q)-1|<C q ^{\frac{1}{2}}$.
\item\label{it:asum.3} $|h_{\Def_{G,n}}(\Z_q/\mathfrak{m}_q^m)-h_{\Def_{G,n}}( \mathbb{F}_q)|< C q ^{-1}.$
\item\label{it:asum.4} $|h_{G}(\Z_q/\mathfrak{m}_q^m)-1|<C q ^{-\frac{1}{2}}$.
\end{enumerate}
\item $G$ is smooth over $\Spec(S^{-1}\Z)$.
\end{itemize}
After a finite extension of scalars, $G_{\mathbb{Q}}$ is isomorphic to a simply connected Chevalley group. Thus, after enlarging $S$, we can assume that $G_{\mathbb{F}_p}$ is simply connected and semisimple, for all $p\notin S$.
In addition, \eqref{it:asum.4} implies that, after a further enlargement of $S$, we can assume that, for any $q\in \mathcal{P}_S$,
\begin{enumerate}[(a)]
\setcounter{enumi}{5}
\item\label{it:asum.6} $h_{G}(\Z_q/\mathfrak{m}_q^m)>\frac{1}{2}$.
\end{enumerate}
For any $m$,
\begin{equation} \label{eq:zeta.h}
\zeta_{G(\mathbb{Z}_q/ \mathfrak{m}_q ^m)}(2n-2)=\frac{|\Def_{G,n}(\mathbb{Z}_q/ \mathfrak{m}_q ^m)|}{|G^{2n-1}(\mathbb{Z}_q/ \mathfrak{m}_q ^m)|}=\frac{h_{\Def_{G,n}}(\mathbb{Z}_q/ \mathfrak{m}_q ^m)}{h_{G^{2n-1}}(\mathbb{Z}_q/ \mathfrak{m}_q ^m)}=\frac{h_{\Def_{G,n}}(\mathbb{Z}_q/ \mathfrak{m}_q ^m)}{h_{G^{2n-1}}(\mathbb{F}_q)},
\end{equation}
where the first equality follows from the Frobenius formula (Theorem \ref{thm:Frobenius}), the second follows from the fact that $\dim \Def_{G_\Q,n}=\dim G_\Q^{2n-1}$ (\cite[Corollary 4.1.4]{AA}), and the third follows from the fact that $G$ is smooth over $O$ and Hensel's lemma.

Using \eqref{eq:zeta.h} with $m=1$ and \eqref{it:asum.1.5},\eqref{it:asum.4}, we get that $\zeta_{G(\mathbb{F}_q)}(2n-2) \rightarrow 1$. Together with \cite[Theorem 3.1]{AKOV} and the fact that $G_{\mathbb{F}_q}$ is semisimple and simply connected, after enlarging $C$ we can assume that for every $q\in \mathcal{P}_S$,

\begin{enumerate}[(a)]
\setcounter{enumi}{6}
\item\label{it:asum.5} $|\zeta_{G(\mathbb{F}_q)}(2n-2)-1|<C q ^{-1}$.
\end{enumerate}

Fix $q\in \mathcal{P}_S$. Let $O$ be either $\mathbb{Z}_q$ or $\mathbb{F}_q[[t]]$, and let $\mathfrak{m} \subset O$ be the maximal ideal. By \eqref{it:asum.1}, \eqref{it:asum.2}, and the Frobenius formula,
\begin{equation}
\zeta_{G(O/ \mathfrak{m} ^m)}(2n-2)=\frac{|\Def_{G,n}(O/ \mathfrak{m} ^m)|}{|G^{2n-1}(O / \mathfrak{m} ^m)|}=\frac{|\Def_{G,n}(\mathbb{Z}_q/ \mathfrak{m}_q ^m)|}{|G^{2n-1}(\mathbb{Z}_q / \mathfrak{m}_q ^m)|}=\zeta_{G(\mathbb{Z}_q/ \mathfrak{m}_q ^m)}(2n-2).
\end{equation}
Thus, Equation \eqref{eq:zeta.h} implies that
\begin{multline*}
|\zeta_{G(O/ \mathfrak{m} ^m)}(2n-2)-1|\leq |\zeta_{G(\mathbb{Z}_q/ \mathfrak{m}_q ^m)}(2n-2)-\zeta_{G(\F_q)}(2n-2)|+|\zeta_{G(\F_q)}(2n-2)-1| =\\
=\frac{|h_{\Def_{G,n}}(\mathbb{Z}_q/ \mathfrak{m}_q ^m)-h_{\Def_{G,n}}(\F_q )|}{h_{G^{2n-1}}(\mathbb{F}_q)}+|\zeta_{G(\F_q)}(2n-2)-1|. \end{multline*}

Using \eqref{it:asum.3}, \eqref{it:asum.6}, \eqref{it:asum.5},  we obtain that:
$$
|\zeta_{G(O/ \mathfrak{m} ^m)}(2n-2)-1|\leq(2^{2n-1}+1)C q ^{-1}.
$$

Since each continuous representation of a pro-finite group is locally constant, we get that the series defining $\zeta_{G(O)}(2n-2)$ equals the limit $$\lim_{m \rightarrow \infty} \zeta_{G(O/ \mathfrak{m} ^m)}(2n-2),$$ and the result follows.

\subsection{Representation zeta functions of adelic groups and the proof of Theorem \ref{thm:introduction.main}} \label{ssec:pf.main}

\begin{proof} [Proof of Theorem \ref{thm:introduction.main}] Theorem \ref{thm:zeta.minus.1} implies that there is a finite set $S$ of prime numbers and a positive integer $D$ such that, for every $q\in \mathcal{P}_S$,
\begin{equation} \label{eq:zeta.main}
\zeta_{G(\mathbb{Z}_q)}(2n-2)-1=\zeta_{G(\mathbb{F}_q[[t]])}(2n-2)-1 \leq D q ^{-1}.
\end{equation}
Let $C>\max S$, let $k$ be a global field of characteristic greater than $C$, let $T$ be a finite set of places of $k$ containing all archimedean ones, and let $n$ be such that $\Def_{G_\mathbb{Q},n}$ has rational singularities. We will show that $\zeta_{G \left( \widehat{O_{k,T}} \right) }(2n-2+\epsilon)$ converges, for every positive $\epsilon$.

For any non-archimedean place $v$, let $O_v$ be the ring of integers of the completion $k_v$, and let $|v|$ be the size of the residue field of $O_v$. We have that $\zeta_{G \left( \widehat{O_{k,T}}\right)}(s)=\prod_{v \notin T} \zeta_{G(O_v)}(s)$, by which we mean, in particular, that the left hand side converges absolutely if and only if each of the terms on the right hand side converges absolutely, and their product converges absolutely. By \cite[Theorem IV]{AA} and the assumptions, the series $\zeta_{G(O_v)}(2n-2)$ converges for every $v$. Thus, it is enough to prove that $\prod_{v \notin T'} \zeta_{G(O_v)}(2n-2+\epsilon)$ converges, for some finite set of places $T' \supset T$.

For almost all places $v$, we have
\begin{equation} \label{eq:O.main}
O_v\cong \mathbb{Z}_{|v|} \textrm{ or } O_v\cong\mathbb{F}_{|v|}[[t]].
\end{equation}

By \cite[Proposition 4.6]{LM}, there is a constant $\delta >0$ such that, for almost all places $v$, the minimal dimension of a non-trivial representation of $G(O_v)$ is at least $|v|^{\delta}$. It follows that
\begin{equation} \label{eq:rep.main}
\zeta_{G(O_v)}(2n-2+\epsilon)-1 \leq |v|^{-\epsilon \delta}\left( \zeta_{G(O_v)}(2n-2)-1\right).
\end{equation}

Let $T' \supset T$ be a finite set of places such that \eqref{eq:zeta.main}, \eqref{eq:O.main}, and \eqref{eq:rep.main} hold for any $v\notin T'$. We have

\[
\prod_{v\notin T'} \zeta_{G(O_v)}(2n-2+\epsilon) \leq \prod_{v \notin T'} \left( 1+D |v| ^{-1-\epsilon \delta} \right) \leq \prod_{v \notin T'} \left( 1+|v| ^{-1-\epsilon \delta} \right)^D \leq \prod_{v\notin T'} \left( 1-|v|^{-1-\epsilon \delta} \right) ^{-D}\leq
\]

\[
\leq \left( \prod_{v\textrm{ non-archimedean}} \left( 1-|v|^{-1-\epsilon \delta} \right) ^{-1}  \right) ^D.
\]
The last product is the product expansion of the Dedekind zeta function of $k$ at the point $1+\epsilon \delta$, and it is well-known that it converges absolutely.

\end{proof}

\subsection{Representation zeta functions of arithmetic groups and the proof of Theorem \ref{thm:introduction.bound}} \label{ssec:pf.bound}

Let us first recall the notion of the Congruence Subgroup Property.
\begin{definition} \label{defn:CSP} Let $k$ be a global field, and let $T$ be a finite set of places of $k$ containing all archimedean ones. We say that $G(O_{k,T})$ has the Congruence Subgroup Property (CSP for short) if the canonical map
\[
\eta : \widehat{G(O_{k,T})} \rightarrow G \left( \widehat{ O_{k,T} } \right) 
\]
has a finite kernel. Here, the domain of $\eta$ is the pro-finite completion of $G(O_{k.T})$.
\end{definition}

By our assumptions on $G$ and the Strong Approximation Theorem (\cite[Chapter 7]{PR}), the map $\eta$ is always surjective. It is known that, if $\rk_{k} G_k \geq 2$, then $G(O_{k,T})$ satisfies CSP. More generally, a conjecture of Serre asserts that the group $G(O_{k,T})$ has CSP whenever $\sum_{v\in T}\rk_{k_v} G \geq 2$ and $\rk_{k_v} G \geq 1$ for any finite place $v\in T$. This conjecture is known in many cases; see, for example, \cite{Ra}.

%

Theorem \ref{thm:introduction.bound}
follows from Theorem \ref{thm:introduction.main} and the following theorem

\begin{theorem} Let $k$ be a global field and let $T$ be a finite set of places of $k$ containing all the archimedean ones. Assume that $G(O_{k,T})$ satisfies CSP, then $\alpha(G(O_{k,T}))=\alpha \left( G(\widehat{O_{k,T}}) \right)$.
\end{theorem}

\begin{proof} If the characteristic of $k$ is non-zero, the claim follows from \cite[Theorem 3.3]{LL} and \cite[Lemma 2.2]{LM}. If the characteristic of $k$ is zero, the claim follows from \cite[Theorem 1.4]{AKOV}.
\end{proof}

\subsection{The Igusa zeta function and proof of Theorem \ref{thm:igu}} \label{ssec:pf.thm.igu}
The following lemma is standard:
\begin{lemma} \label{lem:Dirichlet}
$ $
\begin{enumerate}
\item \label{it:Dirichlet.1} Let $a_n$ be a sequence and assume that there is a constant $C>0$ such that $\frac{1}{C} < a_n < C$ for $n> C$. Then the series $f(s):=\sum a_n p^{-ns}$ converges absolutely for $Re(s)>0$ and diverges for $s=0$.
\item \label{it:Dirichlet.2} Suppose in addition that $f(s)$ coincides (in its domain of absolute convergence) with a rational function of $p^{-s}$, which we continue to denote by $f(s)$. Then $0$ is a simple pole of $f(s)$ and all the poles on the imaginary axis are simple.
\item \label{it:Dirichlet.3} Let $a_{n,p}$ be a sequence indexed by a positive integer $n$ and a prime $p$, and let $\delta,C >0$. Suppose that $|a_{n,p}-1|<Cp^{-\delta}$.  Let $Z(s):=\prod_p(1+\sum_n a_{n,p} p^{-ns})$. Then:
\begin{enumerate}
\item $Z(s)$ converges absolutely for $Re(s)>1$.
\item If $\zeta(s)$ denotes the Riemann zeta function, then $Z(s)/\zeta(s)$ can be analytically continued for $Re(s)>1-\delta$.
\end{enumerate}
\end{enumerate}
\end{lemma}

Claim \ref{item:P.p.pole} of Theorem \ref{thm:igu} follows from Lemma \ref{lem:Dirichlet}(\ref{it:Dirichlet.1},\ref{it:Dirichlet.2}), Theorem \ref{thm:intro.count}, and the rationality of $\mathcal{P}_{X,p}(s)$. Claim \ref{item:P.zeta} of Theorem \ref{thm:igu} follows from Lemma \ref{lem:Dirichlet}\eqref{it:Dirichlet.3} and Theorem \ref{thm:intro.count}. Claim \ref{it:3} of Theorem \ref{thm:igu} follows from claim \ref{item:P.zeta} and \cite[Theorem 4.20]{dSG}.

Finally, claim \ref{item:Z.zeta} of Theorem \ref{thm:igu} follows from Lemma \ref{lem:Dirichlet}\eqref{it:Dirichlet.3}, Theorem \ref{thm:intro.count}, and the fact that
\[
\mathcal{Z}_{X,p}(s)=\sum_n \left( h_X(\mathbb{Z} / p^n)-\frac{h_X(\mathbb{Z} / p^{n+1})}{p} \right) p ^{-n(s+1)}.
\]

\begin{remark} \label{rem:l.c.i.} In view of Remark \ref{rem:l.c.i.h.bdd}, Lemma \ref{lem:Dirichlet} also implies that, if $X$ is local complete intersection, then one can take $S$ to be empty in Theorem \ref{thm:igu}.
\end{remark}


\begin{thebibliography}{KKMS73}

\bibitem[AA]{AA} Aizenbud, A.; Avni, N.; {\em Representation growth and rational singularities of the moduli space of local systems}. \href{http://arxiv.org/abs/1307.0371}{Arxiv 1307.0371}, to appear in Inventiones Mathematicae.
\bibitem[AKOV]{AKOV} Avni, N.; Klopsch, B.; Onn, U.; Voll, C.; {\em Arithmetic Groups, Base Change, and Representation Growth}. \href{http://http://arxiv.org/abs/1110.6092}{Arxiv 1110.6092}.
\bibitem[Avn11]{A} Avni, N.; {\em Arithmetic groups have rational representation growth.} Ann. of Math. 174 (2011), 1009--1056.
\bibitem[BLR90]{NeronModels} Bosch, S.; Lutkebohmert, W.; Raynaud, M.; {\em Neron models}. Ergebnisse der Mathematik und ihrer Grenzgebiete (3) [Results in Mathematics and Related Areas (3)], 21. Springer-Verlag, Berlin, 1990.
 \bibitem[dSG00]{dSG} du Sautoy, M. P. F.; Grunewald, F.; {\em Analytic properties of zeta functions and subgroup growth}. Annals of Math. 152 (2000), 793--833.
\bibitem[Gro66]{EGA4}  Grothendieck A.; \emph{Elements de geometrie algebrique IV} Publications mathematiques de l' I.H.E.S., tome 28 (1966).
\bibitem[Har77]{Har} Hartshorne, R.; \emph{Algebraic geometry}, Springer-Verlag, New York   (1977).
\bibitem[HK06]{HK} Hrushovski, E.; Kazhdan. D.; {\em Integration in valued fields}. Algebraic geometry and number theory, volume 253 of Progr. Math., Birkh\"auser Boston, Boston, MA, 2006.
\bibitem[Jai06]{Jai} Jaikin--Zapirain A.; {\em Zeta function of representations of compact p-adic analytic groups.}  J. Amer. Math. Soc. 19 (2006), 91--118.
\bibitem[KKMS73]{KKMS} Kempf, G.; Knudsen, F.; Mumford D.; Saint-Donat, B.; \emph{Toroidal Embeddings I.} Springer Lecture Notes. no. 339 (1973)
\bibitem[LL08]{LL} Larsen, M.; Lubotzky, A.; {\em Representation growth of linear groups}. J. Eur. Math. Soc. 10 (2008), 351--390.
\bibitem[LM04]{LM} Lubotzky, A.; Martin, B.; {\em Polynomial representation growth and the congruence subgroup problem.} Israel J. Math. 144 (2004), 293--316.
\bibitem[LS05]{LS} Liebeck, M.W.; Shalev, A.; {\em Character degrees and random walks in finite groups of Lie type.} Proc. London Math. Soc. 90 (2005), 61--86.
\bibitem[LW54]{LW} Lang, S.; Weil, A.; {\em Number of points of varieties in finite fields}. Amer. J. Math. 76 (1954), 819--827.
\bibitem[Mus01]{Mus} Mustata, M.; {\em Jet schemes of locally complete intersection canonical singularities.} Invent. Math. 145 (2001), no. 3, 397-424.
\bibitem[Mus]{Mus_zeta} Mustata, M.; {Zeta functions in algebraic geometry}. Lecture notes. Available at \url{http://www.math.lsa.umich.edu/~mmustata/zeta_book.pdf}.
\bibitem[PR94]{PR} Platonov, V.; Rapinchuk, A.; {\em Algebraic groups and number theory.} Pure and Applied Mathematics, 139. Academic Press, Inc., Boston, MA, 1994.
\bibitem[Rag04]{Ra} Raghunathan, M.S.; {\em The congruence subgroup problem.}  Proc.\ Indian Acad.\ Sci.\ Math.\ 114 (2004), no. 4, 299--308.
\bibitem[Ser64]{Se} Serre, J.P. \emph{Lie Algebras and Lie Groups.} Lecture Notes in Mathematics, {\bf 1500}. Springer-Verlag, New York, 1964.

\end{thebibliography}
\end{document}